\documentclass[11pt]{article}

\usepackage[english]{babel}
\usepackage{amsmath,amsfonts,amssymb,amsthm}
\usepackage{latexsym}
\usepackage{makeidx}
\usepackage{color}
\usepackage{vmargin}
\usepackage{enumerate}
\setmarginsrb{20mm}{20mm}{20mm}{20mm}%
            {5mm}{5mm}{5mm}{10mm}

\numberwithin{equation}{section}

\numberwithin{table}{section}

\usepackage{accents}
\newlength{\dhatheight}

\usepackage{scalerel,stackengine}
\stackMath
\newcommand\reallywidehat[1]{%
\savestack{\tmpbox}{\stretchto{%
  \scaleto{%
    \scalerel*[\widthof{\ensuremath{#1}}]{\kern-.6pt\bigwedge\kern-.6pt}%
    {\rule[-\textheight/2]{1ex}{\textheight}}
  }{\textheight}%
}{0.5ex}}%
\stackon[1pt]{#1}{\tmpbox}%
}
\usepackage{caption}

\usepackage[all]{xy}

\theoremstyle{plain}
\newtheorem{defn}{Definition}[section]
\newtheorem{lem}[defn]{Lemma}
\newtheorem{thm}[defn]{Theorem}

\newtheorem{prop}[defn]{Proposition}

\theoremstyle{remark}

\newtheorem{rem}[defn]{Remark}

\global\def\ca#1{\mathcal{#1}}

\newcommand{\EE}{\mathbb E}
\newcommand{\FF}{\mathbb F}
\newcommand{\LL}{\mathbb L}
\newcommand{\K}{\mathbb K}
\newcommand{\N}{\mathbb N}

\newcommand{\B}{\ca{B}}
\newcommand{\BV}{\B(V)}

\newcommand{\LLC}{\mathrm{LLC}}
\newcommand{\LC}{\mathrm{LC}}
\def\Vect{\mathrm{Vect}}

\def\Flow{\mathrm{Flow}}

\def\f{\phi}
\def\ent{\mathrm{ent}}

\def\Hom{\mathrm{CHom}}

\def\res{\mathrm{res}}
\def\Res{\mathrm{Res}}
\def\idn{\mathrm{ind}}
\def\Idn{\mathrm{Ind}}

\title{Topological entropy for locally linearly compact vector spaces\\ and field extensions\thanks{This work was supported by EPSRC Grant N007328/1 Soluble Groups and Cohomology.}}
\begin{document}
\date{}
\author{Ilaria Castellano\\{\footnotesize {\tt ilaria.castellano88@gmail.com}} \\ {\footnotesize University of Southampton}\\ {\footnotesize Building 54, Salisbury Road - 50171BJ Southampton (UK)}} 

\maketitle
\begin{abstract}
 Let $\K$ be a discrete field and $(V,\f)$ a flow over the category of locally linearly compact $\K$-spaces. Here we give the formulas to compute the topological entropy of $(V,\f)$ subject to the extension or the restriction of scalars.  
\end{abstract}
\section{Introduction} 
 Usually, an entropy function over a category $\mathcal{C}$  is an invariant 
$$ \mathrm{h}\colon \Flow(\ca{C})\to \N\cup\{\infty\},$$
of the category $\Flow(\ca{C})$ whose objects are the flows\footnote{A flow in $\mathcal{C}$ is a pair $(C,\gamma)$ consisting of an object $C$ of $\ca{C}$ and an endomorphism $\gamma$ of $C$.} $(C,\gamma)$ over $\mathcal{C}$. 
The first appearance of {\it topological entropy} was in 1965 when Adler, Konheim and McAndrew \cite{AKM} defined it for continuous self-maps of compact spaces. Afterwards, Bowen \cite{Bowen} and Hood \cite{Hood} brought a notion of topological entropy into the world of uniformly continuous self-maps on uniform spaces (also known as {\it uniform entropy}) opening up  the possibility for computing entropy for any given continuous endomorphism $\f\colon G \to G$ of a topological group $G$. E.g., profinite groups and locally compact groups. 

Interest in the locally compact case  is currently very high. Not only because locally compact groups are omnipresent in several areas of mathematics but also because there exists a strict relation between the topological entropy $h_{top}$ and the scale function defined in \cite{Willis} whenever one considers endomorphisms of totally disconnected locally compact groups (see \cite{BDGB,GBV} for  details).  By analogy with the topological entropy $h_{top}$ for totally disconnected locally compact groups, the topological entropy $\ent^*$  has been introduced in the category ${}_\K\LLC$ of locally linearly compact $\K$-spaces \cite{CGB} whenever $\K$ is a discrete field. The main motivation for studying such an entropy function is  to reach a better understanding of $h_{top}$ by means of the more rigid case of locally linearly compact  spaces.  

Recall that a topological vector space $V$ over a discrete field $\K$ is said to be {\it locally linearly compact} if $V$ admits a local basis at 0 consisting of linearly compact open $\K$-subspaces. In particular, a topological vector space $E$ is {linearly compact} when
\begin{enumerate}[(LC1)]
 \item it is a Hausdorff space in which there is a fundamental system of neighbourhoods of 0 consisting of linear subspaces of $E$;
\item any filter base on $E$ consisting of closed linear varieties has a non-empty intersection.
\end{enumerate}

Notice that any locally linearly compact $\K$-space $V$ is clearly a totally disconnected locally compact abelian group - whenever $\K$ is a finite field - and
$$h_{top}(V,\f) = \ent^*(V,\f) \cdot \log |\K|,$$
for every continuous endomorphism $\f\colon V\to V$ (see \cite[Proposition~3.9]{CGB}).

In this paper, we see that both the procedures of restriction and extension of scalars continue to be available in the context of locally linearly compact vector spaces (see \S~\ref{ss:changeLLC}). Namely, given a finite field extension $\FF\leq\K$, there exist functors
$$\Res^\K_\FF(-)\colon{}_\K\LLC\to{}_\FF\LLC\quad\text{and}\quad\Idn^\K_\FF(-)\colon{}_\FF\LLC\to{}_\K\LLC$$
forming an adjunction by
$$\Hom_\K(\Idn^\K_\FF(-),-)\cong\Hom_\FF(-,\Res^\K_\FF(-)).$$

Therefore it is reasonable to ask whether the topological entropy is effected by the change of scalar fields and the question is addressed by the formulas
$$\ent_\FF^*(\Res^\K_{\FF}(V,\f))=[\K:\FF]\cdot\ent_\K^*(V,\f)\quad\text{and}\quad\ent_\LL^*(\Idn_\K^{\LL}(V,\f))=\ent_\K^*(V,\f),$$
that are proved in Theorem~\ref{thm:main} for $\FF\leq\K\leq\LL$ finite field extensions. Moreover, the condition on finite extension is necessary.
There are no surprises in either the statement or the proof of Theorem~\ref{thm:main}, but its content does not appear anywhere else.  Thus this paper is intended to be just a complement to \cite{IC,CGB} where $\ent^*$ for locally linearly compact $\K$-spaces has been introduced and studied. 

\section{Preliminaries and basic properties} Throughout, the topology on arbitrary fields will always be the discrete topology. We summarise here some of the  properties of linearly compact vector spaces and the topological entropy $\ent^*$ (see \cite{CGB} for references). We will use them in the rest of the paper sometimes with no previous acknowledgement.

\subsection{Linearly compact vector spaces}  
A topological $\K$-space is said to be {\it linearly topologized} if it admits a local basis at 0 of $\K$-linear subspaces.
For linearly topologized $\K$-spaces $V, V'$ and $W$ such that $W\leq V$, the following hold:
\begin{enumerate}[(P1)]
\item If $W$ is linearly compact, then $W$ is closed in $V$.
\item If $V$ is linearly compact and $W$ is closed, then $W$ is linearly compact.
\item Linear compactness is preserved by continuous homomorphisms.
\item If $V$ is discrete, then $V$ is linearly compact if and only if $V$ has finite dimension.
\item If $W$ is closed, then $V$ is linearly compact if and only if $W$ and $V/W$ are linearly compact.
\item The direct product of linearly compact $\K$-spaces is linearly compact.
\item An inverse limit of linearly compact $\K$-spaces is linearly compact.
\item If $V$ is linearly compact, then $V$ is complete.
\item Let $V$ be linearly compact. Every continuous $\K$-linear $f\colon V\to V'$ is open onto its image. 
\item Linearly compact  $\K$-spaces satisfy Lefschetz duality.
\item Every linearly compact $\K$-space is topologically a direct product of copies of $\K$.
\end{enumerate}
\subsection{Complete tensor product and profinite plagiarism}
Given  linearly compact $\K$-spaces $V$ and $W$, a linearly compact $\K$-space $T$ (together with a bilinear map $b\colon V\times W\to T$) is a {\it complete tensor product} of $V$ and $W$ over $\K$ if it satisfies the following universal property: for every linearly compact $\K$-space $Z$ and every continuous bilinear map $\beta\colon V\times X\to Z$ there exists a continuous bilinear map $\hat\beta\colon T\to Z$ such that
\begin{equation}
\xymatrix{
V\times X\ar[r]^-b\ar[d]_-{\beta}&T\ar[dl]^-{\hat\beta}\\
Z}
\end{equation}
commutes. It is easy to see that if the complete tensor product exists, it is unique up to isomorphism. We denote it by $V\widehat\otimes_\K W$.
\begin{prop}\label{prop:ctp}
With the above notation, the complete tensor product  $V\widehat\otimes_\K W$ exists. In fact, if
$$V=\varprojlim_{i\in I} V_i\quad\text{and}\quad W=\varprojlim_{j\in J}W_j,$$
where each $V_i$ (respectively, $W_j$) is a finite-dimensional $\K$-space endowed with the discrete topology, then
$$ V\widehat\otimes_\K W= \varprojlim_{i\in I,j\in J}(V_i\otimes_\K W_j),$$
where $V_i\otimes_\K W_j$ is the usual tensor product as abstract $\K$-spaces. In particular, $V\widehat\otimes_\K W$ is the (linearly compact) completion of $V \otimes_\K W$, where $V \otimes_\K W$ has the topology for which a fundamental system of neighbourhoods of 0 are the kernels of the natural maps
$$V\otimes_\K W\to  V_i \otimes_\K W_ j\quad( i \in I , j \in J).$$
\end{prop}
\begin{proof} Basically, we use the strategy developed in the proof of \cite[Lemma~5.5.1]{RZ}. For 
\begin{equation} 
V\times W\cong_{top}\varprojlim_{i \in I , j \in J}  (V_i \times  W_ j),
\end{equation}
there exists a canonical continuous $\K$-bilinear map
\begin{equation}
\iota\colon V\times W\to \varprojlim_{i\in I,j\in J}(V_i\otimes_\K W_j).
\end{equation}
Set $v\hat\otimes w:=\iota(v,w)$ for any $(v,w)\in V\times W$. Since every linearly compact $\K$-space is the inverse limit of its finite-dimensional $\K$-quotients, it suffices to check the universal property only for an arbitrary $\K$-space, say $F$, of finite dimension. Suppose  $\beta\colon V\times W\to F$ to be  continuous and $\K$-bilinear. For $\dim_\K(F)<\infty$, there is a pair of indices $(i,j)\in I\times J$ and a continuous $\K$-bilinear map $\beta_{ij}\colon V_i\times Wj\to F$ such that the diagram
\begin{equation}
\xymatrix{
V\times W\ar[r]^-\beta\ar[dr]&F\\
&V_i\times W_j\ar@{-->}[u]_{\beta_{ij}}}
\end{equation}
commutes\footnote{One can check that the dual factorisation does always exist  in the dual category ${}_\K\Vect$}. Now one uses the universal property of $V_i\otimes_\K W_j$ to produce a continuous $\K$-bilinear map $\hat\beta_{ij}\colon V_i\otimes_\K W_j\to F$ such that $\hat\beta_{ij}(v_i\otimes w_j)=\beta(v_i,w_j)$. Finally, define $\beta\colon V\widehat\otimes_\K W\to F$ to be 
\begin{equation}
\xymatrix{V\widehat\otimes_\K W\ar[r]\ar@/^2pc/[rr]^\beta& V_i\otimes_\K W_j\ar[r]_{\qquad\hat\beta_{ij}} &F.}
\end{equation}
\end{proof}
Moreover, the right-exact covariant functor $V\widehat\otimes_\K -\colon{}_\K\LC\to{}_\K\LC$ is additive and satisfies
\begin{equation}\label{eq:ctp}
V\widehat\otimes_\K \K\cong_{top} V,
\end{equation}
where the isomorphism is natural in $V$. Clearly, similar properties hold for $-\widehat\otimes_\K W$.
\begin{rem}
Since the complete tensor product is unique up to isomorphism (when it does exist), the construction above does not depend on the (inverse limit)-representation of $V$ and $W$.
\end{rem}
\begin{rem} The tensor products here are merely vector spaces, but complete tensor products can be defined for more general object, e.g.,  linearly compact modules over linearly compact commutative rings.
\end{rem}
\subsection{Topological entropy on ${}_\K\LLC$}
Let $(V,\f)$ be a flow over ${}_\K\LLC$ and denote by $\B(V)$ the collection of all linearly compact open $\K$-subspaces of $V$. Then the {\it topological entropy} of $(V,\f)$ is defined to be
\begin{equation}\label{eq:top}
\ent_\K^*(V,\f)=\sup_{U\in\BV}H_\K^*(\f,U),
\end{equation}
where $H_\K^*(\f,U)=\lim_{n\to\infty}\frac{1}{n}\dim_\K(U/(U\cap\f^{-1}U\cap\ldots\cap\f^{-n+1}U))$. Moreover, the linearly compact open $\K$-space $U\cap\f^{-1}U\cap\ldots\cap\f^{-n+1}U$ is called {\it $n$-cotrajectory of $\f$ in $U$} and it is denoted by $C^\K_n(\f,U)$.
\begin{enumerate}[eP1]
\item ({\it Invariance under conjugation}) For every topological isomorphism $\alpha\colon V\to W$  of locally linearly compact $\K$-spaces, 
$\ent_\K^*(W,\alpha\phi\alpha^{-1})=\ent_\K^*(V,\phi).$
\item ({\it Monotonicity}) For a closed $\K$-subspace  $W$ of $V$ such that $\f(W)\leq W$, 
 $$\ent_\K^*(V,\f)\geq \max\{\ent_\K^*(W,\f\restriction_W),\ent_\K^*(V/W,\overline\f)\}$$
 where $\overline\f\colon V/W\to V/W$ is induced by $\f$,
 \item ({\it Logarithmic law}) If $k\in\N$, then 
 $\ent_\K^*(V,\phi^k)=k \cdot \ent_\K^*(V,\phi)$
\item ({\it Continuity on inverse limits}) Let $\{W_i\mid i\in I\}$ be a directed system (for inverse inclusion) of closed $\K$-subspaces of $V$ such that $\f(W_i)\leq W_i$. If $V=\varprojlim V/W_i$, then $\ent_\K^*(\phi)=\sup_{i\in I} \ent_\K^*(\overline\phi_{W_i})$, where any $\overline\f_{W_i}\colon V/W_i\to V/W_i$ is the continuous endomorphism induced by $\f$.
\item ({\it Change of basis}) For every local basis $\B\subseteq\BV$ at 0, one has $\ent_\K^*(V,\f)=\sup_{U\in\B}H_\K^*(\f,U)$.
\end{enumerate}
Notice that the topological entropy over ${}_\K\LLC$ was originally denoted in \cite{CGB} simply by $\ent^*$. Here we need to keep in mind the field we are working on.
\section{Change of fields}\label{s:change}
Firstly, we face the case of linearly compact vector spaces and then we deduce from it the general case concerning locally linearly compact spaces.
\subsection{Restriction of scalars over linearly compact vector spaces}
 
Let $\FF\leq \K$ be a field extension of finite degree $[\K:\FF]\in\N$. Since any $\K$-space can be regarded as $\FF$-space via the inclusion $\FF\hookrightarrow\K$, one has a functor of abstract vector spaces
\begin{equation}
\res^\K_{\FF}(-)\colon{}_\K\Vect\to{}_{\FF}\Vect,
\end{equation}
which is usually called {\it restriction of scalars}. Now let $V$ be a linearly compact $\K$-space, i.e., $V$ is topologically a direct product of (discrete) 1-dimensional $\K$-spaces (endowed with the Tychonoff topology). Since  restriction  commutes with products, the underlying abstract $\FF$-space is
\begin{equation}
\res^\K_{\FF}(V)=\prod\res^\K_{\FF}(\K)\cong\prod\FF^{[\K:\FF]}.
\end{equation}
From the topological point of view,  the abstract $\FF$-space $\res^\K_{\FF}(V)$ comes equipped with the product topology of discrete finite-dimensional $\FF$-spaces. Therefore  $V$ is a linearly compact $\FF$-space. Thus one also obtains a {\it restriction functor}
\begin{equation}\label{eq:resLC}
\Res^\K_{\FF}(-)\colon{}_\K\LC\to{}_{\FF}\LC.
\end{equation}
for linearly compact vector spaces and field extensions.

Notice that  $[\K:\FF]< \infty$ is a necessary condition, since a discrete linearly compact vector space has to be finite-dimensional.
\subsection{Extension of scalars over linearly compact vector spaces}
Let $\K\leq\LL$ be a finite field extension, i.e., $\LL$ is a linearly compact $\K$-space via multiplication. Firstly, notice that
 $V \otimes_\K W$ is dense in $V\widehat\otimes_\K W$ since the complete tensor product is a topological completion of the abstract tensor product. Indeed $V\widehat\otimes_\K W$ is topologically spanned by the set of elements of the form $v\hat\otimes w$. Therefore one can define a structure of topological $\LL$-space on the linearly compact $\K$-space $\LL\widehat\otimes_\K V,$
by extending the natural action 
\begin{equation}
l\cdot(l'\hat\otimes v):=(ll')\hat\otimes v,\quad\text{for all}\ l,l'\in\LL, v\in V. 
\end{equation}
We denote this topological $\LL$-space by $\Idn^\LL_\K(V)$.
\begin{prop} For every linearly compact $\K$-space $V$ and a finite extension $\K\leq\LL$, the induced $\LL$-space
 $\Idn^\LL_\K(V)$ is topologically isomorphic to 
 \begin{enumerate}
 \item  $\LL\otimes_\K V$ with the topology for which a local basis at 0 are the kernels of the natural $\LL$-maps
$$\LL\otimes_\K V\to  \LL\otimes_\K V_ i,\quad i \in I,$$
where $V=\varprojlim_{i\in I} V_i$ for some family $\{V_i\}$ of finite-dimensional $\K$-spaces.
\item $\prod_J\LL$ with Tychonoff topology, where $V$ has been regarded as $\prod_J\K$.
\end{enumerate}
In particular, $\Idn^\LL_\K(V)$ is a linearly compact $\LL$-space. 
\end{prop}
\begin{proof} Since $\LL\cong\K^n$ for some $n$, (1) follows by \eqref{eq:ctp} and Proposition~\ref{prop:ctp}.
 
Finally, (2) follows by (1) and the fact that
$$V\cong_{top}\varprojlim_{F\in\mathcal{F}(J)}\prod_{j\in F}\K,$$
where $\mathcal{F}(J)$ is the directed set of all finite subsets of $J$.
\end{proof}
Therefore it is well-defined the {\it induction functor}
\begin{equation}\label{eq:idnLC}
\Idn_\K^{\LL}(-)\colon{}_{\K}\LC\to{}_{\LL}\LC,
\end{equation}
 whenever $\LL$ is an extension of $\K$ of finite degree. Moreover, the  induction functor $\Idn^\LL_\K(-)$ turns out to be left-adjoint to the restriction functor $\res^\LL_\K(-)$ as usual, i.e., $\LL\widehat\otimes_\K V$ satisfies the following universal property: for every linearly compact $\LL$-space $W$ and every continuous $\K$-linear map $f\colon V\to W$ there exists a unique continuous $\LL$-linear map $\hat f\colon \LL\widehat\otimes_\K V\to W$ such that
 $$
 \xymatrix{
V\ar[r]\ar[d]_{f}&\LL\widehat\otimes_\K V\ar@{-->}[dl]^{\hat f}\\
W
}$$
commutes.
\subsection{Restriction and Induction over ${}_\K\LLC$}\label{ss:changeLLC} Let $\FF\leq\K\leq\LL$ be finite field extensions. Here we define the restriction functor 
\begin{equation}\label{eq:resLLC}
\Res^\K_{\FF}(-)\colon{}_\K\LLC\to{}_{\FF}\LLC.
\end{equation}
and the induction functor 
\begin{equation}\label{eq:idnLLC}
\Idn_\K^{\LL}(-)\colon{}_{\K}\LLC\to{}_{\LL}\mathrm{LLC},
\end{equation}
over the category of locally linearly compact vector spaces relying on the construction given above for linearly compact vector spaces. With abuse of notation, those functors shall still be denoted by $\Res^*_\bullet(-)$ and $\Idn^*_\bullet(-)$. It unlikely causes confusion since the functors defined here coincide with \eqref{eq:resLC} and \eqref{eq:idnLC} over the corresponding subcategory of linearly compact vector spaces.

 Let $V$ be an object in ${}_\K\LLC$. The locally linearly compact $\FF$-space $\Res^\K_\FF(V)$ is defined to be the abstract $\FF$-space $V$ endowed with the topology that is  locally generated at 0 by
\begin{equation}
\{\Res^\K_\FF(U)\mid U\ \text{linearly compact open in}\ V\},
\end{equation}
(compare with \eqref{eq:resLC}). Analogously, $\Idn^\LL_\K(V)$ shall be the abstract $\LL$-space $\LL\otimes_\FF V$ together with  
\begin{equation}
\{\Idn^\LL_\K(U)\mid U\ \text{linearly compact open in}\ V\},
\end{equation}
where each $\Idn^\LL_\K(U)$ (compare with \eqref{eq:idnLC}) can be  identified with a $\K$-subspace of $\LL\otimes_\FF V$ since $\LL$ is free over $\K$. One can easily check that the functors so defined are additive and form the adjunction required, i.e.,
\begin{equation}
\textup{CHom}_\LL(\Idn^\LL_\K(V),W)\cong\textup{CHom}_\K(V,\Res^\LL_\K(W)).
\end{equation}
Since a functor $F\colon {}_\K\LLC\to{}_\bullet\LLC$ induces a functor of flows, we shall simplify the notation as following:
\begin{enumerate}
\item $C^\bullet_n(F(\f,U)):=C^\bullet_n(F(\f),F(U))$;
\item $H^*_\bullet(F(\f,U)):=H^*_\bullet(F(\f),F(U))$;
\item $\ent^*_\bullet(F(V,\f)):=\ent^*_\bullet(F(V),F(\f))$;
\end{enumerate}
where $\bullet\in\{\FF,\K,\LL\}$.
\section{Topological entropy after induction and restriction}
\subsection{Linearly compact case}
 Let $\FF\leq\K\leq\LL$ be field extensions such that $[\LL:\FF]<\infty$. Let $V$ be a linearly compact vector space and let $${}_V\beta\colon \prod_{n=0}^\infty V\to \prod_{n=0}^\infty V,\quad (v_n)_{n\in\N}\mapsto (v_{n+1})_{n\in\N},$$ denote the left Bernoulli shift of $V$. When $V= \K$ then ${}_\K\beta$ is the one-dimensional left Bernoulli shift that has topological entropy $=1$. Clearly,  its image under restriction is nothing but the ($[\K:\FF]$-dimensional) left Bernoulli shift of $\Res^\K_{\FF}(V)=\prod_{n=0}^\infty \FF^{[\K:\FF]}$. Therefore,
$$\ent_\FF^*(\Res^\K_{\FF}(V, {}_\K\beta))=\ent^*_\FF(\prod_{n=0}^\infty \FF^{[\K:\FF]},{}_{\FF^{[\K:\FF]}}\beta)=[\K:\FF]=[\K:\FF]\cdot\ent_\K^*(V,{}_\K\beta),$$
by \cite[Example~3.13(b)]{CGB}.
On the other hand, the induction gives rise to the (1-dimensional) left Bernoulli shift on $\Idn_\K^{\LL}(V)=\prod_{n=0}^\infty\LL$, i.e.,
$$\ent_\LL^*(\Idn_\K^{\LL}(V,{}_\K\beta))=\ent^*_\LL(\prod_{n=0}^\infty\LL,{}_\LL\beta)=1=\ent_\K^*(V,{}_\K\beta).$$
The following result shows that the formulas above hold for every flow $(V,\phi)$ over ${}_\K\LC$.
\begin{prop}\label{prop:main}
Let $\FF\leq\K\leq\LL$ such that $[\LL:\FF]<\infty$. For every flow $(V,\f)$ over ${}_\K\LC$ one has
$$\ent_\FF^*(\Res^\K_{\FF}(V,\f))=[\K:\FF]\cdot\ent_\K^*(V,\f)\quad\text{and}\quad\ent_\LL^*(\idn_\K^{\LL}(V,\f))=\ent_\K^*(V,\f).$$
\end{prop}
For a field $\EE$, the collection of all open $\EE$-subspaces $\prod A_i$ of $\prod\EE$ such that 
\begin{enumerate}[(G1)]
\item $A_i=\EE$ for all but finitely many indices,
\item $A_i=0$ otherwise;
\end{enumerate}
 form a neighbourhood basis at 0 for the Tychonoff topology on $\prod\EE$. We will refer to such a basis as  {\it good basis} of $\prod\EE$ and its elements will be said to be {\it good}.

\begin{proof}[Proof of Proposition~\ref{prop:main}]
For an arbitrary flow $(V,\f)$ over ${}_\K\LC$, let $U=\prod A_i$ be a good $\K$-subspace of $V$.  Thus
\begin{align}
H_\FF^*(\Res^\K_{\FF}(\f,U))&=\lim_{n\to\infty}\frac{1}{n}\cdot\mathrm{codim}_\FF\,C^\FF_n(\Res^\K_{\FF}(\f,U))=\\
&=\lim_{n\to\infty}\frac{[\K:\FF]}{n}\cdot\mathrm{codim}_\K\,C^\K_n(\f,U)=[\K:\FF]\cdot H_\K^*(\f,U).
\end{align}
 and notice that $\Res^\K_\FF(U)$ is a good $\FF$-subspace of $V$. Since  good subspaces form a neighbourhood basis at 0, one has that  
\begin{equation}
[\K:\FF]\cdot\ent_\K^*(V,\f)\leq\ent_\FF^*(\Res^\K_{\FF}(V,\f)).
\end{equation}
Finally, the equality holds since every good $\FF$-subspace of $\Res^\K_{\FF}(V)$ can be realised as the restriction of a good  $\K$-subspace of $V$. In order to prove the second part of the statement, notice that
$$\dim_\K(-)=\dim_\LL(\Idn_\K^\LL(-))\quad\text{and}\quad C^\LL_n(\Idn^\LL_{\K}(\f,U))=\Idn_\K^\LL(C^\K_n(\phi,U)),\ \text{($U\in\mathcal{B}(V), n\in\N$)}.$$
Therefore,
\begin{align}
H^*(\Idn^\LL_{\K}(\f,U))&=\lim_{n\to\infty}\frac{1}{n}\cdot\mathrm{codim}_\LL\,C^\LL_n(\Idn^\LL_{\K}(\f,U))=\lim_{n\to\infty}\frac{1}{n}\cdot\mathrm{codim}_\LL\,\Idn_\K^\LL(C_n(\phi,U))= \footnote{Since $\LL\otimes_\K-$
 is an exact functor, the sequence $0\to \K\otimes_\FF A\to\K\otimes_\FF B\to\K\otimes_\FF (A/B)\to0$ is exact.}\\
&=\lim_{n\to\infty}\frac{1}{n}\cdot\mathrm{codim}_\K\,C_n(\f,U)= H^*(\f,U).
 \end{align}
Same reasoning as above yields $\ent^*(\idn^\LL_{\K}(V,\f))=\ent^*(V,\f).$
\end{proof}
\begin{rem} By latter result, it follows that 
$$\ent_\K^*(\Idn_\FF^\K(\Res_\FF^\K(V,\phi)))=[\K:\FF]\cdot\ent_\K^*(V,\phi)=\ent_\K^*(\Res_\K^\LL(\Idn_\K^\LL(V,\phi))).$$
This gives back an intuitive method to generate examples of topological flows with some finite topological entropy $n\in\N$. E.g., let $(V,{}_\K\beta)$ be the 1-dimensional left Bernoulli shift defined above. One easily generates the $n$-dimensional left Bernoulli shift over $\K$ by combining induction and restriction with respect to a field extension of $\K$ of degree $n$.
\end{rem}
\subsection{General case}

Recall that every locally linearly compact $\K$-space $V$ can be split into a topological direct sum of a linearly compact open $\K$-space $V_c$ and a discrete $\K$-space $V_d$; namely, $V\cong_{top} V_c\oplus V_d$. Such a decomposition is available whenever we consider a linearly compact open subspace $U$ of $V$. Indeed, $V\cong_{top} U\oplus V/U$.

Now let $(V,\f)$ be a flow over ${}_\K\LLC$ . Any decomposition $V\cong_{top} V_c\oplus V_d$ induces a decomposition
\begin{equation}
\f=
\begin{pmatrix}
  \f_{cc} & \f_{dc} \\
  \f_{cd} & \f_{dd}
\end{pmatrix},
\end{equation}
where $\phi_{\bullet*}:V_\bullet\to V_*$ is the composition $\phi_{\bullet*}=p_*\circ\phi\circ\iota_\bullet$ for $\bullet, *\in\{c,d\}$.
Therefore, $\f_{\bullet*}$ is continuous being composition of continuous homomorphisms.

\begin{lem} Let $\FF\leq\K\leq\LL$ be finite field extensions and $V\cong_{top} V_c\oplus V_d$ a locally linearly compact $\K$-space together with a continuous endomorphism $\f\colon V\to V$. Thus
\begin{enumerate}
\item $\Res^\K_\FF(V)$ admits a decomposition such that
$$\Res^\K_\FF(V)_c=\Res^\K_\FF(V_c)\quad\text{and}\quad\Res^\K_\FF(\phi)_{cc}=\Res^\K_\FF(\phi_{cc}).$$
\item $\Idn^\LL_\K(V)$ admits a decomposition such that
$$\Idn_\K^\LL(V)_c=\Idn_\K^\LL(V_c)\quad\text{and}\quad\Idn^\LL_\K(\phi)_{cc}=\Idn^\LL_\FF(\phi_{cc}).$$
\end{enumerate}
\end{lem}
\begin{proof} Notice that $\Res^\K_\FF(V)$ (or $\Idn^\LL_\K(V)$) can be obtained by restricting (or extending) $V_c$ and $V_d$ first in the corresponding categories and then summing them up in ${}_\K\LLC$. Indeed both restriction functor  and induction functor are additive.
\end{proof}
\begin{thm}\label{thm:main}
Let $\FF\leq\K\leq\LL$ such that $[\LL:\FF]<\infty$. For every flow $(V,\f)$ over ${}_\K\LLC$ one has
$$\ent_\FF^*(\Res^\K_{\FF}(V,\f))=[\K:\FF]\cdot\ent_\K^*(V,\f)\quad\text{and}\quad\ent_\LL^*(\Idn_\K^{\LL}(V,\f))=\ent_\K^*(V,\f).$$
\end{thm}
\begin{proof}
By the previous lemma, it follows directly from Proposition~\ref{prop:main} and the reduction to linearly compact vector spaces devolopped in \cite{CGB}.
\end{proof}

\end{document}